\documentclass[a4paper,12pt]{amsart}
\usepackage[a4paper,margin=0.7in]{geometry}
\usepackage{amsmath,amssymb,amsfonts,amsthm}
\usepackage{amsrefs}
\usepackage{mathtools}
\usepackage{tikz}
\usepackage[]{hyperref}
\usepackage{mathtools}
\usepackage{thmtools}

\newtheorem{theorem}{Theorem}
\newtheorem{lemma}[theorem]{Lemma}
\newtheorem{proposition}[theorem]{Proposition}

\newtheorem{corollary}[theorem]{Corollary}
\newtheorem*{corollary*}{Corollary}
\newtheorem{remark}[theorem]{Remark}

\DeclareMathOperator{\re}{Re}

\newcommand{\sgn}{\operatorname{sgn}}

\newcommand{\ol}[1]{\overline{#1}}
\newcommand{\sm}{\smallsetminus}

\newcommand{\Ww}{{W}}
\newcommand{\e}{\mathbf{e}}

\newcommand{\Q}{{\mathbb{Q}}}

\newcommand{\C}{{\mathbb{C}}}
\newcommand{\HH}{{\mathbb{H}}}
\DeclareMathOperator{\SL}{SL}

\title[Convolution identities for divisor sums and modular forms]{Convolution identities for divisor sums\\ and modular forms}
\author{Ksenia Fedosova, Kim Klinger-Logan, and Danylo Radchenko}

\newcommand{\arXiv}[1]{arXiv:\href{https://arxiv.org/abs/#1}{\texttt{#1}}}

\makeatletter
\renewcommand{\sectionautorefname}{\S{}\@gobble}
\makeatother
\def\equationautorefname~#1\null{(#1)\null}

\begin{document}

\maketitle
\begin{abstract} We prove exact identities for convolution sums of divisor functions of the form 
$\sum_{n_1 \in \mathbb{Z} \smallsetminus \{0,n\}}
\varphi(n_1,n-n_1)\sigma_{2m_1}(n_1)\sigma_{2m_2}(n-n_1)$
where $\varphi(n_1,n_2)$ is a Laurent polynomial with logarithms for which the sum is absolutely convergent. 
Such identities are motivated by computations in string theory and prove and generalize a conjecture of Chester, Green, Pufu, Wang, and Wen from \cite{CGPWW}. Originally, it was suspected that such sums, suitably extended to $n_1\in\{0,n\}$ should vanish, but in this paper we find that in general they give Fourier coefficients of holomorphic cusp forms.
\end{abstract}

\section{Introduction}
\label{sec:intro}

In this paper we establish an identity giving a relationship between convolution sums of divisor functions $\sigma_r(n):=\sum_{d|n}d^r$  with $r\in2\mathbb{Z}_{\geq 0}$ and Fourier coefficients of Hecke eigenforms. For example, our main result implies that for   $n \neq 0$, 
\begin{equation}\label{eq:tauid}
\sum_{\stackrel{n_1,n_2 \in \mathbb{Z} \smallsetminus \{0\}}{n_1+n_2=n}}   \psi(n_1, n_2)\sigma_2(n_1)\sigma_2(n_2) =-  (\zeta(2) n^2 + 420 \zeta'(-2) )\sigma_2(n) - \frac{75 L(\Delta, 6) \tau(|n|)}{8 L(\Delta, 5)|n|^3},   
\end{equation}
where $\psi(n_1,n_2) = \tfrac{1}{2} (\tilde{\psi}(n_1,n_2)+\tilde{\psi}(n_2, n_1))$ for $\tilde{\psi}(n_1,n_2)$ defined as 
\begin{equation}\label{eq:psi} 
\begin{split} 
\frac{1}{(n_1+n_2)^3}
&\left(\frac{n_{1}^{5}}{n_2^2} + \frac{35n_{1}^{4}}{n_2} - 1099 n_{1}^{3} + 1575 n_{1}^{2}n_{2}  + (420 n_{1}^{3} - 2100 n_{1}^{2} n_{2})\log |\tfrac{n_1}{n_2}| \vphantom{\frac{n_{1}^{5}}{n_2^2}} \right), \nonumber
\end{split}    
\end{equation}
and $L(\Delta, s)$ is the $L$-function of the weight 12 cusp form $\Delta(z) = \sum_{n \ge 1} \tau(n) e^{2 \pi i n z}$. Such an identity is unexpected, and as far as the authors are aware, the only known relationship between divisor functions and the Ramanujan $\tau$ function involves finite sums of odd index divisor functions~\cite{Ram}. Generally, shifted convolution sums are of number theoretic interest due to their connection to moments of and subconvexity bounds for $L$-functions \cites{BTB, Blomer, CFKRS}.
However, identities such as \eqref{eq:tauid} would be difficult to discover outside of their natural context and, in this case, the investigation of the particular weighted sums is motivated by string theory. 
	
Specifically, sums of the form \eqref{eq:tauid} appear as part of the low energy expansion of the 4-graviton scattering amplitude as well as related calculations in the $\mathcal{N}=4$ Super-Yang-Mills (SYM) gauge theory via the AdS/CFT correspondence \cite{ACDGW, CGPWW, DT, GMRV2010,GMV2015, GRV2010}. 
On the one hand, the appearance of holomorphic cusp forms in this context is unanticipated as they do not appear in corresponding localized computations \cite{ACDGW}. On the other hand, when computing the full integrated correlator, the exact identity established in Theorem \ref{thm:main1} allows one to see that these cusp forms {\it exactly} cancel. This cancellation suggests that the large-$N$ expansion of certain integrals of the correlator of super conformal primary operators in the $\mathcal{N}=4$ stress tensor multiplet can be written as lattice sums~\cite{ACDGW}.


Our work was originally motivated by a conjecture from string theory of Chester, Green, Pufu, Wang, and Wen in \cite[Section C.1(a)]{CGPWW} that a particular shifted convolution sum vanishes, and Theorem \ref{thm:main1}  proves this conjecture. Explicitly, their conjecture can be written as
\begin{equation}\label{eq:conj}
\begin{split}
\displaystyle\sum_{\stackrel{n_1, n_2 \in \mathbb{Z} \smallsetminus \{ 0\}}{n_1+n_2=n}} \varphi(n_1,n_2) & \sigma_2(n_1) \sigma_2(n_2) 
=
\left( \frac{\zeta(2) n^2}{2} + 30 \zeta'(-2)\right)\sigma_2(n) ,
\end{split}
\end{equation}
where 
\begin{equation} 
\begin{split}
\varphi(n_1,n_2)= & -\frac{n_1^2}{4 n_2^2}-\frac{7 n_1}{2 n_2}-\frac{n_2^2}{4 n_1^2}-\frac{7 n_2}{2 n_1}+\frac{47}{2} 
+\left(15-\frac{30 n_1}{n_1+n_2}\right) \log \left|\frac{n_1}{n_2}\right|.
\end{split}
\end{equation}
 
Note that, unlike~\eqref{eq:tauid}, there is no term involving Fourier coefficients of a cusp form in~\eqref{eq:conj}. This conjecture arose from the fact that a constant multiple of the summation (plus the negation of the right side) in \eqref{eq:conj} appears in the homogeneous part
of the Fourier expansion of a translation invariant solution to the equation
\begin{equation}\label{eq:DE}
(\Delta-12)f(z) = -\left(2\zeta(3) E_{3/2}(z)\right)^2,
\end{equation}
where $\Delta = y^2\left(\partial_x^2+\partial_y^2\right)$ on $\operatorname{SL}_2(\mathbb{Z})\backslash \operatorname{SL}_2(\mathbb{R})/\operatorname{SO}_2(\mathbb{R})=\Gamma\backslash\mathbb{H}$ and $E_s(z)$ is the non-holomorphic Eisenstein series
\begin{equation*}E_{s}(z)=\sum_{\gamma\in (B\cap \Gamma)\backslash \Gamma} \text{Im}(\gamma z)^{s}, \quad \text{for Re}(s)>1\,
\end{equation*}
for $B$ the Borel subgroup in $\SL_2(\mathbb{R})$ fixing $\infty\in\overline{\mathbb{H}}$.
Such solutions $f(z)$ to \eqref{eq:DE} give the  $D^6\mathcal{R}^4$ coefficient of the low energy expansion of the 4-graviton scattering amplitude in 10-dimensional type IIB string theory \cite{GMRV2010,GMV2015, GRV2010}. For physical reasons, it was expected that the homogenous solution should vanish. 
	
In \cite{KMR} a formal argument was given for \eqref{eq:conj}; however, the argument relied on evaluating a double Dirichlet series outside its region of convergence, and the authors were unable to make this argument rigorous. Beyond looking for a rigorous proof of the conjecture, it is natural to ask if other convolution sums similar to \eqref{eq:conj} might also hold. Explicitly, we can generate a family of convolution sums via computing homogeneous parts 
of the Fourier expansion of the solutions to differential equations of the form
\begin{equation}\label{eq:DE_general}
(\Delta-s(s+1))f(z) = E_{a}(z)E_{b}(z)
\end{equation}
for other values $a,b$, and $s$ in $\mathbb{R}$ on $\Gamma\backslash\mathbb{H}$. Solutions to \eqref{eq:DE_general} appear in calculations in $\mathcal{N}=4$ SYM gauge theory \cite{DT} and physicists wondered whether the corresponding sums also vanish, giving idenities similar to \eqref{eq:conj}. Numerical evidence told a surprisingly different story.  In fact, instead of vanishing, these sums yield Fourier coefficients of modular forms as we will see in the statement of Theorem~\ref{thm:main1}.

\label{sec:main}
To state our main result precisely, recall the definition of Jacobi functions of the second kind (see~\cite[p.~172]{E1} for $x\in\mathbb{C}\smallsetminus[-1,1]$,
\cite[\S4.61]{Sz}):
	\begin{equation}
\label{eq:jacobi-def}
Q_{d}^{(\alpha,\beta)}(x)=\frac{(x-1)^{-\alpha}(x+1)^{-\beta}}{2^{d+1}}\int_{-1}^{1}\frac{(1-t)^{d+\alpha}(1+t)^{d+\beta}dt}{(x-t)^{d+1}},
\end{equation}
which we extend to $x\in(-1,1)$ by setting
$Q_{d}^{(\alpha,\beta)}(x)=\frac12(Q_{d}^{(\alpha,\beta)}(x+i0)+Q_{d}^{(\alpha,\beta)}(x-i0))$. We will only discuss $Q_d^{(\alpha,\beta)}$ for $\alpha,\beta\in\mathbb{Z}_{\ge0}$ (more precisely, for $\alpha, \beta \in 2 \mathbb{Z}_{\ge 0}$), in which case the expression on the right in~\eqref{eq:jacobi-def} is single-valued in the cut plane $\C\sm[-1,1]$ and defines an elementary function (see Proposition~\ref{prop4} below).
	
	
\begin{theorem} \label{thm:main1}
Let $d\in  \mathbb{Z}_{>0}$ and $r_1,r_2\in 2\mathbb{Z}_{\geq 0}$. Then for any $n\in \mathbb{Z}_{>0}$,  
\begin{equation} 
\label{eq:mainid}
   \begin{split}
& \sum_{\substack{n_1,n_2\in\mathbb{Z}\sm\{0\}\\n_1+n_2=n}}
Q_{d}^{(r_1,r_2)}\Big(\frac{n_2-n_1}{n_1+n_2}\Big)\sigma_{r_1}(n_1)\sigma_{r_2}(n_2)  = (-1)^dZ_{d}^{(r_1,r_2)}(n)\sigma_{r_1}(n) - Z_{d}^{(r_2,r_1)}(n)\sigma_{r_2}(n) + \frac{a_n}{n^d}\,,
\end{split}
\end{equation}
		where
\begin{align*}Z_{d}^{(\alpha,\beta)}(n) = \left\{
\begin{array}{lr}\frac{(\beta-1)!(\alpha+d)!}{2(\alpha+\beta+d)!}\zeta(\beta)n^{\beta} + 
\binom{d+\beta}{d}\frac{\zeta'(-\beta)}{2}\,, & \beta\neq 0, \\[10pt]
\frac{1}{4}\left(H_{d+\alpha}+H_d-\log \left|4 \pi ^2
				n\right|\right)\,, & \beta=0,
\end{array}\right.
\end{align*}
where $H_d$ is the $d$-th harmonic number
and $h(\tau):=\sum_{m\ge1}a_m q^m$ is a cusp form of weight \[k:=2d+r_1+r_2+2\] on $\mathrm{SL}_2(\mathbb{Z})$, given by $h = \sum_{f}\lambda_f f$, where $f$ runs over 
normalized Hecke eigenforms\footnote{We say that a Hecke eigenform is normalized if its first non-zero Fourier coefficient is equal to 1.} of weight $k$ and level~$1$, and
\[
\lambda_f = \frac{\pi (-1)^{d+r_2/2+1}}{2^{k}}\binom{k-2}{d}
\frac{L^{\star}(f,d+1)L^{\star}(f,r_1+d+1)}{\langle f,f\rangle}\,.\]
Here $\langle f,g\rangle:=\int_{\Gamma\backslash\HH}f(z)\ol{g(z)}y^{k-2}dxdy$ 
denotes the Petersson inner product, and $L^\star(f, \cdot)$ is the completed $L$-function of $f$:
\[L^\star(f,s) = (2 \pi)^{-s} \Gamma(s) L(f,s).\]
\end{theorem}
	
For $n<0$, the identity \eqref{eq:mainid} remains true if we write $\frac{a_{|n|}}{|n|^d}$ instead of $\frac{a_n}{n^d}$.
	

To see how this statement applies to solutions of differential equations of the form~\eqref{eq:DE_general}, we note that for $a,b\in 1/2+\mathbb{Z}_{>0}$ and $s$ large enough, the homogeneous solution to \eqref{eq:DE_general} 
is given by \[\displaystyle \sum_{n\in\mathbb{Z}}\alpha_n \sqrt{y} K_{s+1/2}(2\pi |n|y)e^{2\pi i n x}\] for  $K_{s}(z)$ the modified Bessel function of the second kind  and $\alpha_n$ is a multiple of 
\begin{equation*}
\begin{split}
 \sum_{\substack{n_1,n_2\in\mathbb{Z}\sm\{0\}\\n_1+n_2=n}}&
Q_{d}^{(r_1,r_2)}\Big(\frac{n_2-n_1}{n_1+n_2}\Big)\sigma_{r_1}(n_1)\sigma_{r_2}(n_2)+(-1)^{d+1}Z_{d}^{(r_1,r_2)}(n)\sigma_{r_1}(n) + Z_{d}^{(r_2,r_1)}(n)\sigma_{r_2}(n) ,
\end{split}
\end{equation*}
where $r_1=2a-1$, $r_2=2b-1$, and $d=s+1-a-b$ (see~\cite{FKL}). 

In what remains of the introduction we will discuss other work related to convolution sums of divisor functions. In Section \ref{sec:examples}, we give a corollary of Theorem \ref{thm:main1} which expresses \eqref{eq:mainid} in terms of polynomials and logarithms as opposed to Jacobi functions of the second kind. We will also provide some examples of identities of the form in \eqref{eq:mainid} and discuss the ramifications of Theorem~\ref{thm:main1} in physics. In Section \ref{sec:jacobi-properties}, we will establish properties of the Jacobi function with integer parameters which we later use in the proof of our main result. In Section \ref{sec:proof_main}, we will first prove some integral identities involving Whittaker functions. We then provide a precise statement of the Holomorphic Projection Lemma which we will then use to prove Theorem~\ref{thm:main1}. 
	
\subsection{Related Results}
Before the full theory of modular forms had been developed, Jacobi, Glaisher \cite{Glaisher}, and Ramanujan \cite{Ram} examined sums involving  divisor functions.
The formula 
\begin{equation}\label{Jac}
\sum_{n=1}^{N-1}\sigma_3(n)\sigma_3(N-n) =\frac{1}{120}(\sigma_7(N)-\sigma_3(N))
\end{equation} was attributed to 
Jacobi\footnote{We were unable to find a primary source for this result.} \cite{Motohashi}. 
In 1885, Glaisher gave expressions for the first five powers of the generating series of the sum of divisors function $\sigma_1(n)$ \cite{Glaisher}. Motivated by generalizing \eqref{Jac}, Ramanujan manipulated what would later be known as the Eisenstein series $E_2$, $E_4$, and $E_6$ to study finite convolution sums of odd divisor functions \cite[p. 136--162]{Ram}. 
Later in 1969, Lahiri found identities involving sums of $\sigma_1$ shifted by pentagonal numbers \cite{Lahiri}.
	
Generally, identities similar to \eqref{Jac} involving {\it finite} convolution sums of odd divisor functions can be found using holomorphic Eisenstein series by computing the Fourier coefficients of their products or Rankin-Cohen brackets~\cite[p.\,18\,\&\,56]{Zagier123}.  Many other examples of such formulas were derived in a recent work of O'Sullivan \cite{OS2} using holomorphic projection. However, none of these identities treat even index divisor functions nor infinite sums of them as in \eqref{eq:mainid}.
	
Perhaps more closely related to this work, Diamantis proved that one can express quotients of values of $L$-functions associated to a normalized cusp form in terms of certain shifted convolution sums~\cite[Theorem~1.1]{D}. The formula in \eqref{eq:mainid} gives an explicit form of such an expression. Motohashi gives a representation in terms of spectral data
of a weighted sum of divisor functions ${\sum_{n=1}^\infty \sigma_0(n) \sigma_0(n+f)W(n/f)}$ where $f\geq 1$ and $W~\in~C_0^\infty(\mathbb{R}_{>0})$ \cite[Theorem 3]{Motohashi}.
	
	
\subsection{Corollaries and Examples}\label{sec:examples}
As they appear in string theory, it is not obvious that these divisor sums can be expressed in terms of Jacobi functions. It may be useful to instead think of these weightings as combinations of 
polynomials in $n_j$, $1/n_j,$ 
and $\log|n_j|$ for $j=1,2$.
	
{\corollary{\label{cor}For $r_1,r_2\in 2\mathbb{Z}_{\geq 0}$, $d\in\mathbb{Z}_{>0}$ and $n>0$, 
let  
\begin{equation}
\label{eq:phidef}
\begin{split}
\varphi(n_1,n_2) :=& \sum_{j=-r_1}^{d-1}  A_j n_1^j + \sum_{j=-r_2}^{d-1} B_j n_2^j  + \sum_{j=0}^d \left( C_j n_1^j \log|n_1| + D_j n_2^j \log |n_2| \right)  
\end{split}
\end{equation}
be such that $\varphi(n_1,n-n_1)=O(n_1^{-d-r_1-r_2-1})$ for $n_1 \to \pm \infty$. 
Then for $n_1+n_2 = n$, 
\[\varphi(n_1,n_2) =   \Gamma^{(r_1,r_2)}_d Q_d^{(r_1,r_2)}\left(\frac{n_2-n_1}{n_2+n_1}\right)
\]
and thus
\begin{equation*}
\begin{split}
& \sum_{\substack{n_1+n_2=n\\n_1n_2\ne 0}}  \varphi(n_1,n_2) \sigma_{r_1} (n_1) \sigma_{r_2}(n_2) 
=\Gamma^{(r_1,r_2)}_d \cdot \left[(-1)^dZ_{d}^{(r_1,r_2)}(n)\sigma_{r_1}(n) - Z_{d}^{(r_2,r_1)}(n)\sigma_{r_2}(n) + \frac{a_n}{n^d}\right]
\end{split}    			
\end{equation*}		    
where $Z_{d}^{(\alpha,\beta)}$ and $a_n$ are as defined in Theorem \ref{thm:main1}, and  \[\Gamma^{(r_1,r_2)}_d: = (-1)^{d+1} n^dC_d  \frac{2 d! (r_1+r_2+d)!}{(r_1+r_2+2d)!},\]
where $C_d$ is as in \eqref{eq:phidef}.
}}	
For $r_1=r_2=0$ or $r_1 r_2 \neq 0$, the first two terms on the right hand side of \eqref{eq:mainid} coincide with the ones predicted formally in \cite{FKL2}. 
In the case when there are no cuspforms of weight $k:=2d+r_1+r_2+2$, the predictions given by the formal arguments in \cite{FKL2, KMR} are proven by Theorem \ref{thm:main1}. Explicitly, when specifying $r_1=r_2=2$ and $d=1$, since there are no cusp forms of weight 8, 
Corollary~\ref{cor} proves the identity~\eqref{eq:conj} as originally conjectured in~\cite{CGPWW}. Moreover, Corollary~\ref{cor} proves the conjectures in~\cite{FKL2} by showing that when $r_1=r_2=0$ and $d=1$, we have
\begin{equation}
\begin{split}
 \sum_{\stackrel{n_1, n_2 \in \mathbb{Z} \smallsetminus \{ 0\}}{n_1+n_2=n}}& \sigma_0(n_1)\sigma_0(n_2) \Big[  \frac{n_2-n_1}{n} \log \left|  \frac{n_1}{n_2}  \right|  + 2 \Big]  =(2-\log \left(4 \pi ^2 |n|\right) ) \sigma_0(n),    
\end{split}		
\end{equation}
		and, when $r_1=r_2=0$ and $d=3$, we get that 
		\begin{equation*}
\begin{split}
 \sum_{\stackrel{n_1,n_2 \in \mathbb{Z}\smallsetminus \{ 0 \} }{n_1+n_2=n}} \sigma_0(n_1)   \sigma_0(n_2) \psi_1(n_1,n_2)= \left(11-3\log(4\pi^2|n|)\right)\sigma_0(n),
\end{split}
\end{equation*}
		where $\psi_1(n_1,n_2) $ equals 
\begin{equation*}
\begin{split}
 11 - \frac{60n_1n_2}{n^2} - \frac{3n_1^3-27n_1^2n_2+27n_1n_2^2-3n_2^3}{n^3}\log \left|\frac{n_1}{n_2}\right|
\end{split}
\end{equation*}
		for $n=n_1+n_2$.
		
		Furthermore, when analyzing the homogeneous solution to \eqref{eq:DE_general} for $a=b=3/2$ and $s=5$, a constant multiple of  \[\sum_{\stackrel{n_1,n_2 \in \mathbb{Z} \smallsetminus\{0\}}{n_1+n_2=n}}   \psi(n_1, n_2)\sigma_2(n_1)\sigma_2(n_2)\]
		for  $\psi$ as in \eqref{eq:psi} appears.
		Using Corollary \ref{cor} when $r_1=r_2=2$ and $d=3$, one sees that \eqref{eq:tauid} holds.
		As a final example, in the homogeneous solution to \eqref{eq:DE_general} for $a=3/2$, $b=5/2$ and $s=11$ involves a constant multiple of the sum \begin{align}\label{eq:psi2sum}\sum_{\stackrel{n_1,n_2 \in \mathbb{Z}\smallsetminus \{ 0 \}}{n_1+n_2=n}} \sigma_2(n_1)\sigma_4(n_2) \psi_2(n_1, n_2),
		\end{align} where 
		\begin{align*}
\psi_2(n_1,n_2)= &\, \frac{7106 n_1^7}{n^7}-\frac{22287 n_1^6}{n^6}+\frac{84626 n_1^5}{3 n^5}-\frac{110789n_1^4}{6 n^4}\\
    &+\frac{33286 n_1^3}{5 n^3}-\frac{3893 n_1^2}{3 n^2}+\frac{2614 n_1}{21 n}
			-\frac{1727}{420}+\frac{n}{63n_1}\\ &  +\frac{n^2}{8190n_1^2}
			-\frac{22 n}{63 n_2}  -\frac{11 n^2}{1365n_2^2}
			-\frac{n^3}{4095n_2^3}-\frac{n^4}{180180 n_2^4}
			\\ &
			-\Big(\frac{11n_1^8}{n^8}-\frac{176n_1^7n_2}{n^8}
			+\frac{924n_1^6n_2^2}{n^8}-\frac{2112n_1^5n_2^3}{n^8}
			+\frac{2310n_1^4n_2^4}{n^8} \\
			& \ \ \ \ \ \ \ \ \  -\frac{1232n_1^3n_2^5}{n^8} +\frac{308n_1^2n_2^6}{n^8}-\frac{32n_1n_2^7}{n^8}
			+\frac{n_2^8}{n^8}\Big)\log \Big|\frac{n_1}{n_2}\Big|.
		\end{align*}
		Using  $r_1=2$, $r_2=4$, and $d=8$ in Corollary \ref{cor}, we get that \eqref{eq:psi2sum} is equal to 
		\begin{align}
		\left(\frac{\zeta(4)}{180180}n^4+\frac{33\zeta(5)}{4\pi^4}\right)\sigma_2(n)+\left(-\frac{\zeta(2)}{8190}n^2+\frac{\zeta(3)}{4\pi^2}\right)\sigma_4(n)+\frac{a_n }{n^8},
		\end{align}
		where
  \begin{equation*}
      \begin{split}
        \displaystyle\sum_{n\geq 1} a_n e^{2\pi i n z} =& \,\frac{L(9,f_1)}{168 L(8,f_1)}\left(-29+\frac{3551}{\sqrt{144169}}\right) f_1(z) +\frac{L(9,f_2)}{168 L(8,f_2)}\left(-29-\frac{3551}{\sqrt{144169}}\right) f_2(z)  
      \end{split}
  \end{equation*}
		and 
		\begin{align*}
			f_1(z) &= e^{2\pi i z} + (540-12\sqrt{144169})e^{4\pi i z} + \dots \\
			f_2(z) &= e^{2\pi i z} + (540+12\sqrt{144169})e^{4\pi i z} +  \dots \end{align*}
		are normalized Hecke eigenforms of weight 24.

		One implication of Theorem \ref{thm:main1} is that certain linear combinations of generalized Eisenstein series $\mathcal{E}(s,a,b,z,\bar z)$ (i.e., solutions to \eqref{eq:DE_general}) 
		arising in physics have no cuspidal components. 
		Such linear combinations of generalized Eisenstein series occur when examining the regularized large $N$ expansion of certain integrated correlators in $SU(N)$ $\mathcal{N}=4$ SYM theory.
		Specifically, it was recently understood that the four superconformal primary operators in the $\mathcal{N}=4$ stress tensor multiplet are obtained from derivatives of the partition $Z$ of the mass-deformed $SU(N)$ $\mathcal{N}=4$ SYM theory placed on a squashed four-sphere \cite{BCPW, CP}.
		For the partition function $Z$,  
		the $N^{-3}$-term of $\partial_m^4\log Z|_{m=0,b=1}$   (see \cite[(2.13)]{CGPWW}) is given~by 
		\begin{equation}\label{eq:physN-3}
	\begin{split}
	\alpha_3\, &\mathcal{E}\left(3,\frac{3}{2},\frac{3}{2}, z,\bar z\right)+\sum_{r=5,7,9}\left[ \alpha_r\, \mathcal{E}\left(r,\frac{3}{2},\frac{3}{2}, z,\bar z\right)  +\beta_r\, \mathcal{E}\left(r,\frac{5}{2},\frac{5}{2}, z,\bar z\right)+\gamma_r\,\mathcal{E}\left(r,\frac{7}{2},\frac{3}{2}, z,\bar z\right)\right]    
	\end{split}		
   	\end{equation} 
		where 
		$\alpha_i,\beta_i$ and $\gamma_i$ are all constants defined in \cite[(2.14)]{CGPWW}.  
		
		When $r=5$, we expect the pieces of $\mathcal{E}\left(r,\frac{3}{2},\frac{3}{2}, z,\bar z\right)$, $ \mathcal{E}\left(r,\frac{5}{2},\frac{5}{2}, z,\bar z\right),$ and $\mathcal{E}\left(r,\frac{7}{2},\frac{3}{2}, z,\bar z\right)$ which correspond to the homogeneous solution to \eqref{eq:DE_general} will contain $L$-values and Fourier coefficients of the weight $12$ cusp form\footnote{e.g., in the first case, $r_1 = r_2 = 2, d=3$ and thus $k = 2d+r_1+r_2+2 = 12$.}~$\Delta$. However, the linear combination of these terms appearing in~\eqref{eq:physN-3} vanishes. 
		To see this, let $L^\star(s) := L^\star(\Delta, s)$ and $L(s) := L(\Delta, s)$ and note that $\langle \Delta, \Delta \rangle$ is a common denominator in all terms, and thus we can omit it from the consideration. Moreover, the $Z_d^{(r_1, r_2)}$ terms will vanish as they simply contribute to the cases when $n_1 n_2 =0$. Thus it suffices to consider the sum 
		\begin{align*}
			\frac{4032}{5 \pi ^4} \alpha_5 D(2,2,3)  
			+ \frac{7168}{5 \pi ^2} \beta_5  D(4,4,1) 
			+ \frac{3072}{5 \pi ^2} \gamma_5 D(2,6,1),
		\end{align*}
		for
		\begin{equation*}
		    \begin{split}
		D(r_1, r_2, d) :=\, &(-1)^{d+\frac{r_2}{2}+1} 2^{-(2 d+r_1+r_2+2)}  L^{\star}(d+1)  L^{\star}(d+r_1+1) \binom{2 d+r_1+r_2}{d}
		    \end{split}
		\end{equation*}
and 
		$\displaystyle
		\alpha_5 = - \frac{135}{52 \pi ^3}$, $\displaystyle\beta_5 = - \frac{30375}{832 \pi ^5}$, and $\displaystyle\gamma_5 = - \frac{42525}{832 \pi ^5}
		$ \cite[(2.14)]{CGPWW}.
		After a substitution it suffices to check 
		that 
		\begin{align*}
&-\frac{382725 L(2) L(4)}{53248 \pi ^{13}}-\frac{1148175 L(4) L(6)}{26624 \pi ^{17}} +\frac{3189375 L(2) L(6)}{53248 \pi ^{15}}
\end{align*}
vanishes, that can be done with the help of \cite{Manin}.

		We note that the cuspidal contribution to \eqref{eq:physN-3} also vanishes when $r=7$ and $r=9$ confirming   physical heuristics.
		In fact, in \cite{ACDGW}, the authors found that that for higher order terms in the $1/N $ expansion for the integrated correlator the cuspidal terms cancel as in \eqref{eq:physN-3}, implying that these terms can be represented as lattice sums. This result further suggests that there should be a more optimal choice than generalized Eisenstein series as a basis for such computations.

		
		\section{Jacobi functions with integer parameters}
		\label{sec:jacobi-properties}
		
		We note that in the case $\alpha,\beta, d\in\mathbb{Z}_{\ge0}$, the Jacobi functions $Q_{d}^{(\alpha,\beta)}(x)$ can be expressed in elementary terms and have the following characterization (compare with~\cite[Eq.~(5.7)]{GZ}).
		\begin{proposition}\label{prop4}
			Let  $\alpha,\beta,d\in\mathbb{Z}_{\ge0}$ and let $P_d^{(\alpha, \beta)}$ denote the Jacobi polynomial defined in~\cite[\S 4.1]{Sz}. For $x\in\mathbb{R}\sm\{-1,1\}$, we have
			\begin{equation} \label{eq:Qn_polyform}
				Q_{d}^{(\alpha,\beta)}(x)
				= \frac{(-1)^{\alpha}}{2}P_d^{(\alpha,\beta)}(x)\log\Big|\frac{x+1}{x-1}\Big|
				+\frac{R(x)}{(x-1)^{\alpha}(x+1)^{\beta}},
			\end{equation}
			where  $R\in\Q[x]$ is a polynomial of degree $d+\alpha+\beta-1$.
			Moreover, let $F$ be any function of the form 
			\begin{equation} \label{eq:Qn_characterization}
				F(x) = P(x)\log\Big|\frac{x+1}{x-1}\Big|+\frac{R(x)}{(x-1)^{\alpha}(x+1)^{\beta}}
			\end{equation}
			with $P,R\in\mathbb{R}[x]$, and $P$ of degree $\le d$ such that $F(x) = O(x^{-d-\alpha-\beta-1})$, $x\to\infty$; then $F$ must be a multiple of $Q_d^{(\alpha,\beta)}(x)$.
		\end{proposition}
		\begin{proof}
			The first claim follows from the integral representation (see~\cite[Eq.~(4.61.4)]{Sz})
\begin{equation*}
    \begin{split}
        	Q_{d}^{(\alpha,\beta)}(x)  =&
			\frac{(x-1)^{-\alpha}(x+1)^{-\beta}}{2}  \int_{-1}^{1}\frac{(1-t)^{\alpha}(1+t)^{\beta}P_d^{(\alpha,\beta)}(t)dt}{(x-t)}\,,
    \end{split}
\end{equation*}
			and writing $\int_{-1}^{1}\frac{p(t)dt}{x-t} = p(x)\int_{-1}^{1}\frac{dt}{x-t}-
			\int_{-1}^{1}\frac{p(x)-p(t)}{x-t}dt$. From~\eqref{eq:jacobi-def} it immediately follows that 
			\begin{equation}\label{eq:asymptotic_expansion_for_Q}
				Q_d^{(\alpha,\beta)}(x)\sim Cx^{-d-\alpha-\beta-1}\,,\qquad x\to\infty
			\end{equation}
			for some $C\ne0$, so $Q_d^{(\alpha,\beta)}(x) = O(x^{-d-\alpha-\beta-1})$ as $x\to\infty$. Now assume that~$F$ is given by~\eqref{eq:Qn_characterization} and satisfies $F(x)=O(x^{-d-\alpha-\beta-1})$. Any polynomial can be written as a linear combination of Jacobi polynomials, thus
			\begin{equation}\label{eq:P_decomposition}
				P(x)=\sum_{j=0}^{d}c_jP_j^{(\alpha,\beta)}(x)    
			\end{equation}
			for some $c_j \in \mathbb{R}$. 
			Consider 
			\[G(x) := F(x)-2(-1)^{\alpha}\sum_{j=0}^{d}c_jQ_j^{(\alpha,\beta)}(x)\,,\]
			where $c_j$ are as in \eqref{eq:P_decomposition}. 
			Then on one hand, \eqref{eq:Qn_polyform}, \eqref{eq:Qn_characterization} and \eqref{eq:P_decomposition} imply that $G(x)$ is of the form $(x-1)^{-\alpha}(x+1)^{-\beta}\widetilde{R}(x)$ for some polynomial~$\widetilde{R}$, and on the other hand, the assumption that $F(x) = O(x^{-d-\alpha-\beta-1})$  and~\eqref{eq:asymptotic_expansion_for_Q} imply that $G(x)=O(x^{-\alpha-\beta-1})$. This can only happen if $\widetilde{R}$ vanishes, and hence~$G=0$. Since $Q_j^{(\alpha,\beta)}(x)\sim C_jx^{-j-\alpha-\beta-1}$ for some $C_j\ne0$ and $F(x)=O(x^{-d-\alpha-\beta-1})$, we also see that $c_j=0$ for $j=0,\dots,d-1$, and thus $F$ is a multiple of $Q_d^{(\alpha,\beta)}$, as claimed.
		\end{proof}
		In view of the above characterization, the polynomial $R$ from~\eqref{eq:Qn_polyform} can be computed from the following identity for formal Laurent series in~$X$
\begin{equation*}
\begin{split}
\frac{(1-X^{-1})^{\alpha}(1+X^{-1})^{\beta}}{2} P_d^{(\alpha,\beta)}(X^{-1})& \log\Big(\frac{1+X}{1-X}\Big)= R(X^{-1})+O(X)\,.        
\end{split}
\end{equation*}
We note that each Jacobi function can be represented in terms of the hypergeometric function~$_2F_{1}$. More precisely, from~\cite[Eq.~(4.61.5)]{Sz} and the symmetry 
		\[Q_d^{(\alpha,\beta)}(x)=(-1)^{\alpha+\beta+d+1}Q_{d}^{(\beta,\alpha)}(-x),\] we obtain
		\begin{equation}\label{eq:Q_via_hypergeometric_function}
  \begin{split}
      	Q_d^{(\alpha,\beta)}(x) &= 
			\frac{2^{d+\alpha+\beta}(d+\alpha)!(d+\beta)!}{(2d+\alpha+\beta+1)!(x-1)^{\alpha}(x+1)^{d+\beta+1}} 
			{}_2F_{1}\Big(d+\beta+1,d+1;2d+\alpha+\beta+2;\frac{2}{1+x}\Big)\, ,
  \end{split}
		\end{equation}
		where in order to incorporate the extension of $Q_d^{(\alpha,\beta)}(x)$ to $x \in (-1,1)$, we set for $t>1$
		\begin{equation}\label{eq:extension_of_hypergeometric_functions}
			{}_2F_1(a,b;c;t) := \frac{1}{2}\big({}_2F_1(a,b;c;t+i0)+{}_2F_1(a,b;c;t-i0)\big)\,.
		\end{equation}
		We will use this definition of $Q_d^{(\alpha,\beta)}(x)$ in terms of hypergeometric functions to compute Mellin transforms of Whittaker functions in the following section. Thus we additionally need some results on the Mellin transform of the hypergeometric function ${}_2F_1$. 
		Specifically, by \cite[p. 314, Eq. 2.21.1.2]{Prud},
\begin{equation}\label{eq:prudnikov1}
\begin{split}
	& \int_{0}^{\infty}  \frac{\Gamma(a)\Gamma(b)\Gamma(c-a)\Gamma(c-b)}{\Gamma(c)}{}_2F_{1}(a,b;c;1-x)x^{s-1}dx\\ & \qquad \qquad = \Gamma(s)\Gamma(s-(a+b-c))\Gamma(a-s)\Gamma(b-s)\,,
\end{split}
		\end{equation}
		from which one also gets
\begin{equation}\label{eq:prudnikov2}
\begin{split}
\int_{0}^{\infty}&\frac{\Gamma(a)\Gamma(b)\Gamma(c-a)\Gamma(c-b)}{\Gamma(c)}{}_2F_{1}(a,b;c;1+x)x^{s-1}dx\\ &\qquad \qquad  = \cos(\pi s)\Gamma(s)\Gamma(s-(a+b-c))\Gamma(a-s)\Gamma(b-s)\,,
\end{split}
\end{equation}
		where we define ${}_2F_1(a,b;c;t)$ for $t>1$ as in \eqref{eq:extension_of_hypergeometric_functions}.

		\section{Proof of the main result}
		\label{sec:proof_main}
		Let $W_{\kappa,\mu}$ denote Whittaker's $W$-functions as in \cite[Eq. 13.14.1]{DLMF} or \cite[p.~386]{E2}.
		It will be convenient to extend $W_{\kappa,\mu}$ to a function defined on $\mathbb{R}\sm \{0\}$ as follows
		\begin{equation} \label{eq:Wext}
			W_{\kappa,\mu}(y) := \frac{\Gamma(1/2+\mu-\sgn(y)\kappa)}{\Gamma(1/2+\mu-\kappa)}
			W_{\sgn(y)\kappa,\mu}(|y|)\,, \quad y\ne 0\,.
		\end{equation}
		Next, we will need two results about $W_{\kappa,\mu}$.
		\begin{lemma} \label{lem:Wmellin}
			For $\kappa,\mu\ge0$ and $\mathrm{Re}(s)>\mu-1/2$ we have
			\begin{equation} \label{eq:Wmellin1}
				\int_{0}^{\infty}W_{\kappa,\mu}(t)e^{-t/2}t^{s-1}dt = \frac{\Gamma(s-\mu+1/2)\Gamma(s+\mu+1/2)}{\Gamma(s-\kappa+1)}\,\\
			\end{equation}
			and for $\kappa\ge\re(s)>\mu-1/2$ we have
\begin{equation} \label{eq:Wmellin2}
\begin{split}
\int_{0}^{\infty}&\Ww_{\kappa,\mu}(-t)e^{t/2}t^{s-1}dt  =\frac{\cos(\pi(\kappa-\mu))}{\pi}\Gamma(s-\mu+1/2)\Gamma(s+\mu+1/2)\Gamma(\kappa-s)\,.    
\end{split}
\end{equation}
		\end{lemma}
		\begin{proof}
			The first equation is a special case of~\cite[\S6.9 Eq.~(8)]{E2}. 
			The second equation
			is~\cite[\S6.9 Eq.~(7)]{E2} together with the functional equation for the gamma function. (Note that there is a typo in~\cite[\S6.9 Eq.~(7)]{E2}, compare with \cite[Eq.~(13.23.5)]{DLMF}.)
		\end{proof}

		\begin{lemma} \label{lem:Wconvolution} 
			Let $k_1,k_2,m_1,m_2\in\mathbb{Z}_{\ge0}$ satisfy $m_i\le k_i$ and $m_1+m_2<k_1+k_2$, and set \[\ell:=k_1+k_2-m_1-m_2-1 \quad 
			\text{and} 
			\quad 
			q:=k_1+k_2+m_1+m_2-1.\] Then for any
			$a,b\in\mathbb{R}$ such that $a+b=1$ and $ab\ne 0$
			one has
\begin{equation}
\begin{split}
&\int_{0}^{\infty}W_{k_1,m_1}(a y)W_{k_2,m_2}(b y) e^{-y/2}y^{k_1+k_2-2}dy \\
   & \qquad \qquad  =\frac{2(-1)^{k_1-m_1}q!\ell!}{\pi}|a|^{m_1+1/2}|b|^{m_2+1/2}Q_{\ell}^{(2m_1,2m_2)}(b-a)\,.
\end{split}       
\end{equation}
	
		\end{lemma}
		
		\begin{proof}
			We start by considering the case $a\in(0,1)$ and define  $I_1(x)$  for $x>0$ by 
\begin{equation*}
    \begin{split}
x^{-m_1-1/2} \int_{0}^{\infty}
			W_{k_1,m_1}(xy)W_{k_2,m_2}(y)e^{-(1+x)y/2}y^{k_1+k_2-2}dy\,.        
    \end{split}
\end{equation*}
			In order to prove the lemma in this case, it suffices to show that 
			\begin{equation}\label{eq:I1_1}
				I_1(x)=2(-1)^{k_1-m_1}\ell!q!
				\frac{Q_{\ell}^{(2m_1,2m_2)}(\tfrac{1-x}{1+x})}{\pi(1+x)^{q+1}}\,.
			\end{equation}
			
			Using~\eqref{eq:Wmellin1} and 
			\cite[\S6.1 Eq.~(13)]{E2}) we obtain that for $2m_1<\mathrm{Re}(s)<k_1+k_2+m_1-m_2$, the Mellin transform of $I_1(x)$ is equal to
			\begin{align*}
				\mathcal{M}(I_1,s) 
				 = & \frac{\Gamma(s-2m_1)\Gamma(s)\Gamma(\ell+2m_1-s+1)
					\Gamma(q-s+1)}{\Gamma(s-m_1-k_1+1/2)\Gamma(k_1+m_1-s+1/2)} \\
				 = &\frac{1}{\pi}\cos\left(\pi (s-k_1-m_1)\right)\Gamma(s-2m_1)\Gamma(s)  \Gamma(\ell+2m_1-s+1) \Gamma(q-s+1).
			\end{align*}
			With the help of \eqref{eq:Q_via_hypergeometric_function}, we write 
\begin{equation}\label{eq:C1}
\begin{split}
C_1\cdot \frac{Q^{(2m_1, 2m_2)}_{\ell} (\tfrac{1-x}{1+x})}{(x+1)^{q+1}} 
& =\,C_1\cdot\frac{  \Gamma \left(q-2m_2+1\right) \Gamma \left(q-2m_1+1\right) }{2 \Gamma(2k_1+2k_2)}\\
& \qquad \qquad \times   {}_2F_1\left(q-2m_2+1,q+1;2k_1+2k_2;1+x\right).
\end{split}
\end{equation}
			By \eqref{eq:prudnikov2}, the Mellin transform of \eqref{eq:C1} is equal to
			\begin{align*}
	C_1\cdot&	\frac{\cos (\pi  s) \Gamma (s) \Gamma \left(s-2 m_1\right)}{2 \Gamma \left(\ell+1\right) \Gamma \left(q+1\right)}  \Gamma \left(-s+q-2m_2+1\right) \Gamma \left(-s+q+1\right).
			\end{align*}
			In order for the expression above to coincide with $\mathcal{M}(I_1,s) $, we must require that 
			\begin{align*}
				C_1	 = \frac{2  \Gamma \left(\ell+1\right) \Gamma \left(q+1\right) \cos \left(\pi\left(-k_1-m_1+s\right)\right)}{\pi \cos (\pi  s)}.
			\end{align*}
			Thus, we obtain for $k_1+m_1 \in \mathbb{Z}$ that $I_1(x)$ is equal to 
\begin{align}\label{eq:Wconvolution_goal1}
\frac{2}{\pi }   \Gamma \left(\ell+1\right) \Gamma \left(q+1\right) \cos \left(\pi\left(k_1+m_1\right)\right) \frac{Q^{(2m_1, 2m_2)}_{\ell} (\tfrac{1-x}{1+x})}{(x+1)^{q+1}} .
			\end{align}
			If $\ell, q \in \mathbb{N}_0$ and $k_1+m_1 \equiv k_1-m_1\bmod 2$, then we exactly recover  \eqref{eq:I1_1}.

It remains to prove the result for $a<0$ (the case $b<0$ follows by symmetry). For $x>0$, we define $I_2(x)$ as 
\begin{equation*}
\begin{split}
\int_{0}^{\infty}x^{-m_1-1/2}
\Ww_{k_1,m_1}(-xy)\Ww_{k_2,m_2}(y)e^{-(1-x)y/2}y^{k_1+k_2-2}dy,      
\end{split}
\end{equation*}
and again using Lemma~\ref{lem:Wmellin} and  \cite[\S6.1 Eq.~(13)]{E2} we calculate $\mathcal{M}(I_2,s)$ is equal to 
			\begin{align*}
\frac{ \cos( \pi (k_1-m_1))}{\pi}&\Gamma(s-2m_1)\Gamma(s)\Gamma(k_1+m_1-s+1/2)  \\
& \times  \frac{\Gamma(\ell+2m_1-s+1)\Gamma(q-s+1)}{\Gamma(k_1+m_1-s+1/2)} 
\end{align*}
with the same fundamental strip as before. With the help of \eqref{eq:Q_via_hypergeometric_function}, we write 
\begin{equation}\label{eq:C2}
\begin{split}
C_2\cdot\frac{Q^{(2m_1, 2m_2)}_{\ell} (\tfrac{1+x}{1-x})}{(1-x)^{q+1}} 
&=C_2\cdot\frac{ \Gamma \left(q-2m_2+1\right) \Gamma \left(q-2m_1+1\right)}{2 \Gamma( 2 k_1 + 2 k_2  )} \\
& \qquad \qquad \times \, _2 F_1\left(q-2m_2+1,q+1;2 \left(k_1+k_2\right);1-x\right).
				\end{split}
			\end{equation}
			By \eqref{eq:prudnikov1}, the Mellin transform of \eqref{eq:C2} is equal to
			\begin{align*}
				C_2\cdot\frac{	 \Gamma (s) \Gamma \left(s-2 m_1\right) \Gamma \left(-s+q-2m_2+1\right) \Gamma \left(-s+q+1\right)}{2 \Gamma \left(\ell+1\right) \Gamma \left(q+1\right)}.
			\end{align*}
			In order for the expression above to coincide with $\mathcal{M}(I_2,s) $, we must require that 
			\begin{align*}
				C_2 =\frac{2 }{\pi } \cos \left(\pi  \left(k_1-m_1\right)\right) \Gamma \left(\ell+1\right) \Gamma \left(q+1\right).
			\end{align*}
Thus, $I_2(x) $ is equal to 
\begin{align*}
\frac{2 }{\pi } \cos \left(\pi  \left(k_1-m_1\right)\right) \Gamma \left(\ell+1\right) \Gamma \left(q+1\right) \frac{Q^{(2m_1, 2m_2)}_{\ell} (\tfrac{1+x}{1-x})}{(1-x)^{q+1}}.
			\end{align*}
			If we assume $k_1-m_1 \in \mathbb{Z}$ and $\ell, q \in \mathbb{N}_0$, then we obtain 
			\begin{equation}\label{eq:I2_1}
				I_2(x)=2(-1)^{k_1-m_1}\ell!q!
				\frac{Q_{\ell}^{(2m_1,2m_2)}(\tfrac{1+x}{1-x})}{\pi(1-x)^{q+1}}\,,
			\end{equation}
			which is easily seen to imply the claim of the lemma for $a<0$.
		\end{proof}

		Finally, we recall the holomorphic projection lemma from~\cite{GZ} (we restrict
		to the case of $\SL_2(\mathbb{Z})$).
		
		\begin{lemma}[Holomorphic Projection Lemma]
			Let $\widetilde{\Phi}$ be a non-holomorphic modular form of weight $k>2$ for $\SL_2(\mathbb{Z})$ with a
			Fourier expansion $\widetilde{\Phi}(z) = \sum_{m\in\mathbb{Z}}a_m(y)e^{2\pi i mx}$,
			and suppose that for some $\varepsilon>0$ we have $\widetilde{\Phi}(z)=O(y^{-\varepsilon})$ as $z\to i\infty$. Define
			\[a_m = \frac{(4\pi m)^{k-1}}{(k-2)!}\int_{0}^{\infty}a_m(y)e^{-2\pi my}y^{k-2}dy\,,\qquad m>0\,.\]
			Then the function $\Phi(z)=\sum_{m>0}a_me^{2\pi i m z}$ is a holomorphic cusp form of weight $k$ for $\SL_2(\mathbb{Z})$ and moreover $\langle f,\Phi\rangle=\langle f,\widetilde{\Phi}\rangle$ for all $f\in S_k(\SL_2(\mathbb{Z}))$. \qed
		\end{lemma}
		
\begin{proof}[Proof of Theorem~\ref{thm:main1}] 
With our notation for $\Ww_{\kappa,\mu}$ the completed non-holomorphic Eisenstein series $E^{*}_{2k}(z,s)$, $k\in\mathbb{Z}_{\ge0}$,  has the following Fourier expansion (cf.~\cite[p.~210]{M},~\cite[\S3]{OS})
\begin{equation} \label{eq:eisen}
\begin{split}
    E_{2k}^{*}(z,s) =\, &c_{k,s}(y) +
(-1)^{k}\sum_{n\ne 0}\frac{\sigma_{2s-1}(n)}{|n|^s}\Ww_{k,s-1/2}(4\pi ny)e^{2\pi i nx}\,.
\end{split}
\end{equation}
Here 
\begin{equation*}c_{k,s}(y)= 
\begin{cases}
\theta_{k}(s)y^s + \theta_{k}(1-s)y^{1-s} \quad & 
 \text{for } s>1/2\,,\\
\frac{\Gamma \left(k+\frac{1}{2}\right) }{\sqrt{\pi }}\left(\psi \left(k+\frac{1}{2}\right)+2 \gamma  + \log(y/\pi )\right)\sqrt{y} 
\quad & \text{for } s=1/2\,,    
\end{cases}
\end{equation*}
where $\theta_k(s)=\pi^{-s}\Gamma(s+k)\zeta(2s)$, 
$\psi(z)$ is the digamma function, and $\gamma$ denotes the Euler–Mascheroni constant. Denote $m_i=r_i/2$ and let 
\[\widetilde{\Phi}(z) = E^{*}_{2k_1}(z,m_1+1/2)E^{*}_{2k_2}(z,m_2+1/2)y^{-k_1-k_2},\]
where we choose the integers $k_1$ and $k_2$ to satisfy $2k_1+2k_2=k=r_1+r_2+2d+2$ and $k_i\ge m_i$. Expanding the product of the Fourier expansions of the two Eisenstein series we see that the coefficients of the holomorphic projection of $\widetilde{\Phi}$ can be calculated as 
			\[a_n=\sum_{n_1+n_2=n} a_{n_1,n_2}\,,\]
where for $n_1n_2\ne 0$ we have
\begin{align*}	 a_{n_1,n_2} 
&= \frac{(4\pi n)^{w-1}}{(w-2)!}\frac{(-1)^{k_1+k_2}\sigma_{r_1}(n_1)\sigma_{r_2}(n_2)}{|n_1|^{m_1+1/2}|n_2|^{m_2+1/2}} \int_{0}^{\infty}\Ww_{k_1,m_1}(4\pi n_1y)\Ww_{k_2,m_2}(4\pi n_2y)e^{-2\pi ny}y^{k_1+k_2-2}dy\\
&=2(-1)^{k_2+m_1}(4\pi)^{k_1+k_2}n^{d}\sigma_{r_1}(n_1)\sigma_{r_2}(n_2)  \frac{d!(d+r_1+r_2)!}{\pi(2d+r_1+r_2)!}Q_{d}^{(r_1,r_2)}\Big(\frac{n_2-n_1}{n}\Big)
\end{align*}
by Lemma~\ref{lem:Wconvolution}. The boundary terms can be  calculated from Lemma~\ref{lem:Wmellin} as
\begin{equation}\label{eq:a_n0}
\begin{split}
& a_{n,0}  = (-1)^{k_2}\frac{(4\pi n)^{k-1}}{(k-2)!}\frac{\sigma_{r_1}(n)}{|n|^{m_1+1/2}} \int_{0}^{\infty}\Ww_{k_1,m_1}(4\pi ny)c_{k_2,m_2+1/2}(y)e^{-2\pi ny}y^{k_1+k_2-2}dy,
\end{split}
\end{equation}
that is,
\begin{equation*}
\begin{split}
 a_{n,0}  &= (-1)^{k_2}\frac{ 2^k \pi^{\tfrac{k}{2}-1} \cos \left(\pi  m_2\right)}{\left(k-2\right)!}
n^{d}\sigma_{r_1}(n) \\ &\qquad \qquad \times \Big(n^{r_2} \zeta \left(r_2\right) \Gamma \left(r_2\right) \Gamma \left(d+1\right) \Gamma
\left(\tfrac{k}{2}+m_1-m_2\right)\\ 
& \qquad \qquad \phantom{\times ( ( } +\zeta \left(-r_2\right) \Gamma \left(-r_2\right) \Gamma \left(\tfrac{k}{2}-m_1+m_2\right) \Gamma \left(r_1+r_2+d+1\right)\Big)
\end{split}
\end{equation*}
where $r_2\neq 0$ and, similarly,
\begin{equation}
\begin{split}
 a_{0,n} 
&= (-1)^{k_1} \frac{ 2^k \pi ^{\tfrac{k}{2}-1} \cos \left(\pi 
m_1\right)  }{\left(k-2\right)!}n^{d}\sigma_{r_2}(n)\label{eq:a_0n}\\ 
&\qquad \qquad \times \Big(n^{r_1}
\zeta \left(r_1\right) \Gamma \left(r_1\right) \Gamma
\left(d+1\right) \Gamma
\left(\tfrac{k}{2}-m_1+m_2\right)\\ 
& \qquad \qquad \phantom{\times ( ( }  +\zeta \left(-r_1\right) \Gamma \left(-r_1\right)  \Gamma \left(\tfrac{k}{2}+m_1-m_2\right)\Gamma \left(r_1+r_2+d+1\right)\Big)
\end{split}
\end{equation}
when $r_1\neq 0$. When $r_2=0$, $a_{n,0}$ can be obtained by taking a limit $r_2 \to 0$ in \eqref{eq:a_n0} due to the absolute convergence of the respective integrals: 
\begin{align*}a_{n,0}=&\frac{(-1)^{k_2} 2^{k-1} \pi
^{\tfrac{k}{2}-1} n^{\tfrac{k}{2}-m_1-1} \Gamma
\left(\tfrac{k}{2}-m_1\right) \Gamma
\left(\tfrac{k}{2}+m_1\right)}{\left(k-2\right)!}\\ & \ \ \times
\left(H_{k/2-m_1-1}+H_{k/2+m_1-1}
-\log(4 \pi ^2 n)\right)\sigma_{r_1}(n).
\end{align*}
where $H_d = \psi(d+1)+\gamma$
and
\begin{align*}a_{0,n}=&\, 
 \frac{ 4^{d+m_1+1} \pi ^{d+m_1} \cos \left(\pi    m_1\right) n^{d-m_1}\sigma_0(n)}{\left(2  \left(d+m_1\right)\right)!} 
\\  &\qquad \qquad \times ( \Gamma (d+1)^2 n^{2 m_1} \zeta \left(r_1\right) \Gamma \left(2  m_1\right) 
+ \zeta \left(-r_1\right) \Gamma  \left(-r_1\right) \Gamma \left(d+2  m_1+1\right)^2 )  \end{align*}
			Similarly, 
			when $r_1=0$, $a_{0,n}$ can be obtained by taking a limit $r_1 \to 0$ in \eqref{eq:a_0n}.
			When $r_1=r_2=0$, we obtain 
\begin{equation*}
\begin{split}
a_{0,n} = a_{n,0}& =\frac{2^{2 d+1} (-1)^d \pi ^d n^d \Gamma (d+1)^2 }{(2 d)!}  \left(2 H_d-\log \left(4 \pi ^2	n\right)\right)\sigma_{0}(n).       
\end{split}
\end{equation*}
By the result of Diamantis and O'Sullivan~\cite[Prop.~2.1]{DOS}, for any normalized Hecke eigenform $f\in S_k(\SL_2(\mathbb{Z}))$ we have 
\begin{align*}
 \langle f,\widetilde{\Phi}\rangle & = 2(-1)^{k_2}\pi^kL^{\star}(f,d+1)L^{\star}(f,r_2+d+1)\\
&  = 2(-1)^{k_2+m_2-m_1-d-1}\pi^kL^{\star}(f,d+1)L^{\star}(f,r_1+d+1)\,,
\end{align*}
from which, in combination with the above formulas for $a_{n_1,n_2}$, we recover~\eqref{eq:mainid}.
\end{proof}
\begin{proof}[Proof of Corollary~\ref{cor}]
We recall $n>0$. 
We use the substitution $n_1 = \frac{n}{2} (1- x)$, $n_2= \frac{n}{2} (1+x)$ so that $\varphi(n_1,n_2)$ becomes a function of $n$ and $x$, call it $\varphi_n(x)$. The summation in $\varphi_n$ containing logarithms becomes 
\begin{align*}
\log|\tfrac{x+1}{x-1}| P_{n,2}(x) + \log|x-1| (P_{n,2}(x)+P_{n,3}(x)) ,    
\end{align*}
where $P_{n,2}(x)$ and $P_{n,3}(x)$ are polynomials in $x$ with the degree at most $d$. 
The growth condition $\varphi(n_1,n-n_1)=O(n_1^{-d-r_1-r_2-1})$ for $n_1 \to \pm \infty$ and fixed $n$ implies that $P_{n,2}(x)+P_{n,3}(x)$ vanishes. 
			After substitution, we get 
\[
\varphi_n(x) = \frac{P_{n,1}(x)}{(1-x)^{r_1}(1+x)^{r_2}}   + P_{n,2}(x) \log|\tfrac{x+1}{x-1}|,
\]
where $P_{n,1}(x)$ is a polynomial in $x$ of degree at most $r_1+r_2+d-1$. 
Thus by Proposition \ref{prop4}, the function $\varphi_n(x)$ (considered as a function of $x$ while keeping $n$ fixed) is a constant multiple  of $Q_d^{(r_1,r_2)}(x)$. We recall that the coefficient of $Q_d^{(r_1,r_2)}(\frac{n_2-n_1}{n_1+n_2})$ in front of $\log |n_1|$ is equal to 
\begin{equation*}
\begin{split}
\frac{(-1)^{r_1+1}}{2}&P_d^{(r_1 ,r_2 )}\left(\frac{n_2-n_1}{n_1+n_2}\right) = 
\frac{(-1)^{r_1+1}}{2(n_1+n_2)^d}\sum_{s=0}^d \frac{(-1)^s (d+r_1)!(d+r_2)! n_1^s n_2^{d-s}}{s!(d-s)!(s+r_1)!(d-s+r_2)!}.        
\end{split}
\end{equation*}
Changing the variables $n_2 = n-n_1$, we obtain a polynomial in $\tfrac{n_1}{n}$ with leading coefficient 
\begin{align*}
\frac{(-1)^{r_1+1}}{2} \sum_{s=0}^d (-1)^d &\frac{(d+r_1)!(d+r_2)!}{s!(d-s)!(s+r_1)!(d-s+r_2)!}= \frac{(-1)^{r_1+d+1}}{2} \frac{(r_1 +\beta +2 d)!}{d! (r_1 +r_2 +d)!} \left(\frac{n_1}{n}\right)^d.
\end{align*}
Comparing the leading coefficients of $\varphi_n(x)$ and $Q_d^{(r_1,r_2)}(x)$  we see that 
\[\varphi_n(x) = (-1)^{r_1 + d+1}n^d C_d  \frac{2 d! (r_1+r_2+d)!}{(r_1+r_2+2d)!} Q_d^{(r_1,r_2)}(x),
\]
which fixes the values of $A_j$, $B_j$, $C_j$ and $D_j$ and concludes the proof.
			\end{proof}

\subsection*{Achnowledgements}
The authors would like to thank Daniele Dorigoni, Michael Green, Axel Kleinschmidt, Stephen D. Miller, and Don Zagier for their many helpful conversations. K-L acknowledges support from the NSF through grants DMS-2001909 and DMS-2302309. KF is partially funded by the Deutsche Forschungsgemeinschaft (DFG, German Research Foundation) under Germany's Excellence Strategy EXC 2044-390685587, Mathematics M\"unster: Dynamics-Geometry-Structure. DR acknowledges funding by the European Union (ERC, FourIntExP, 101078782). Views and opinions expressed are those of the author(s) only and do not necessarily reflect those of the European Union or European Research Council (ERC). Neither the European Union nor ERC can be held responsible for them.


\begin{thebibliography}{99}
\bibitem{Glaisher} J.\,W.\,L.\,Glaisher, {\it Expressions for the first five powers of series in which the coefficients are the sums of the divisors of the exponents,} from {\it Mathematical papers, chiefly connected with the q-series in elliptic functions,} Metcalfe, Cambridge (1885).
\end{thebibliography}
\end{document}